\documentclass{amsart}

\usepackage{epsfig}
\usepackage{amsmath}
\usepackage{amssymb}
\usepackage{amscd}
\usepackage{graphicx}
\usepackage{color}

\newtheorem{thm}{Theorem}[section]
\newtheorem{lem}[thm]{Lemma}
\newtheorem{prop}[thm]{Proposition}

\newtheorem{cor}[thm]{Corollary}

\newtheorem{rem}[thm]{Remark}

\def \N {\mathbb N}

\def \Z {\mathbb Z}

\numberwithin{equation}{section}

\begin{document}

\title{chaotic behavior of group actions}

\thanks{* Corresponding author (chiaths.zhang@gmail.com)}

\author{Zhaolong Wang and Guohua Zhang $^*$}

\address{\hskip- \parindent
Zhaolong Wang, School of Mathematical Sciences, Fudan University, Shanghai 200433, China}
\email{11210180006@fudan.edu.cn}

\address{\hskip- \parindent
Guohua Zhang, School of Mathematical Sciences and LMNS, Fudan University and Shanghai Center for Mathematical Sciences, Shanghai 200433, China}
\email{chiaths.zhang@gmail.com}

\begin{abstract}
In this paper we study chaotic behavior of actions of a countable discrete group acting on a compact metric space by self-homeomorphisms.

For actions of a countable discrete group $G$, we introduce local weak mixing and Li-Yorke chaos; and prove that local weak mixing implies Li-Yorke chaos if $G$ is infinite, and positive topological entropy implies local weak mixing if $G$ is an infinite countable discrete amenable group.
Moreover, when considering a shift of finite type for actions of an infinite countable amenable group $G$, if the action has positive topological entropy then its homoclinic equivalence relation is non-trivial, and the converse holds true if additionally $G$ is residually finite and the action contains a dense set of periodic points.
\end{abstract}

\maketitle

\markboth{Chaotic Behavior of Group Actions}{Z. L. Wang and G. H. Zhang}


\section{Introduction}

The topological theory of dynamical systems concentrates on the study of qualitative aspects of dynamics of continuous transformations by various topological tools. Recently, the notion of weakly mixing subsets (of all orders) was introduced for dynamical systems (of $\mathbb{N}$-actions) by Blanchard and Huang in \cite{BlanchardHuang}, which measures local complexity of the structure of orbits. It was proved that every $\mathbb{N}$-action with positive topological entropy has a non-trivial so-called weakly mixing subset (of all orders) \cite[Theorem 4.5]{BlanchardHuang}. For a better understanding of the dynamics over subsets, in \cite{PZstudia} Oprocha and the second author of the present paper introduced the notion of weakly mixing subsets of order $n$ for any given $n\in \mathbb{N}$.
 Generally speaking, a weakly mixing subset $A$ mimics the synchronization of transfer times of continuous transformations, but this synchronization is local, possibly only for open subsets intersecting $A$. As shown by results in a series of papers \cite{PZETDS, PZstudia, PZtopoappli, PZadvance}, though the definition of weakly mixing subsets is analogous to the case of transformations, dynamical properties of weakly mixing subsets are much harder to manage even we consider dynamical systems of $\mathbb{N}$-actions. For example, it is well known that weak mixing of order 2 is equivalent to weak mixing of all orders for a single transformation \cite{Furstenberg}. While, this is no longer true for weakly mixing subsets as shown by \cite{PZETDS, PZstudia}, in fact, for each $n\ge 2$ there are minimal dynamical systems of $\mathbb{N}$-actions which contain non-trivial weakly mixing subsets of order $n$ and all weakly mixing subsets of order $n+ 1$ in the systems are trivial \cite{PZETDS}. There are many other richer dynamics over weakly mixing subsets observed in \cite{PZETDS, PZstudia, PZtopoappli, PZadvance}.

It is very natural to generalize the definition of weakly mixing subsets from $\mathbb{N}$-actions to actions of a countable discrete group acting on a compact metric space.

{\it Throughout the whole paper, $G$ will be a countable discrete group, and $(X, G)$ will be a $G$-system} in the sense that the group $G$ acts on a compact metric space $X$ (with metric $\rho$) as a group of self-homeomorphisms of the space $X$. We shall write $g: X\rightarrow X, x\mapsto g x$; and naturally $(X^m, G), m\in \mathbb{N}$ will be a $G$-system with $g: X^m\rightarrow X^m, (x_1, \cdots, x_m)\mapsto (g x_1, \cdots, g x_m)$ for all $g\in G$ and $x, x_1, \cdots, x_m\in X$.

 For $\emptyset\neq K_1, K_2\subset X$ and $x\in X$ we set $N (K_1, K_2)= \{g\in G: g K_1\cap K_2\neq \emptyset\}$ and $N (x, K_2)= N (\{x\}, K_2)= \{g\in G: g x\in K_2\}$. Let $\emptyset\neq K\subset X$. We say \emph{$K$ trivial} if it is a singleton; else, we say \emph{$K$ non-trivial}.
Let $n\in \mathbb{N}$. Following \cite{BlanchardHuang, PZstudia}, we say that \emph{$K$ is weakly mixing of order $n$} if $N (U_1\cap K, V_1)\cap \cdots\cap N (U_n\cap K, V_n)\neq \emptyset$ whenever all $U_1, V_1, \cdots, U_n, V_n$ are open subsets of $X$ intersecting $K$; and \emph{weakly mixing (of all orders)} if it is weakly mixing of order $m$ for each $m\in \mathbb{N}$. Observe that, $K$ is weakly mixing of order $n$ if and only if the closure of $K$ is weakly mixing of order $n$, and weak mixing of order $n+ 1$ implies weak mixing of order $n$.

Then we can generalize \cite[Theorem 4.5]{BlanchardHuang} to a more general setting if the group $G$ is good enough (in the sense that it is an infinite countable amenable group).

\begin{thm} \label{wmamenable}
Assume that the system $(X, G)$ has positive topological entropy, where $G$ is an infinite countable discrete amenable group. Then it admits a non-trivial weakly mixing subset of all orders. In fact,
$$h_\text{top} (G, X)= \sup \{h_\text{top} (G, K): \emptyset\neq K\subset X\ \text{is weakly mixing of all orders}\}.$$
\end{thm}

In the setting of dynamical systems of $\mathbb{N}$-actions, the notion of weakly mixing subsets always involves complicated dynamics of the given system, for example, the existence of a non-trivial weakly mixing set implies chaos in the sense of Li and Yorke \cite[Definition 4.1]{BlanchardHuang} and \cite[Proposition 5.1]{PZETDS} (in fact, it reflects a much stronger chaos as shown by \cite[Theorem 5.3]{PZETDS}).
Consequently, each dynamical system of an $\mathbb{N}$-action with positive topological entropy is Li-Yorke chaotic, which was observed first by Blanchard, Glasner, Kolyada and Maass in \cite{BlanchardGlasnerKolyadaMaass}.

The definition of Li-Yorke chaos can be generalized to actions of a countable discrete group \cite{KerrLisoficindependence} as follows.
We say that $(x_1, x_2)\in X^2$ is a \emph{Li-Yorke pair} if
$$\limsup_{G\ni g\rightarrow \infty} \rho (g x_1, g x_2)> 0\ \text{and}\ \liminf_{G\ni g\rightarrow \infty} \rho (g x_1, g x_2)= 0,$$
where the limit supremum and limit infimum mean the limits of
$$\sup_{g\in G\setminus F} \rho (g x_1, g x_2)\ \text{and}\ \inf_{g\in G\setminus F} \rho (g x_1, g x_2),$$
respectively, over the net of $F\in \mathcal{F}_G$. Where, we denote by $\mathcal{F}_G$ the set of all non-empty finite subsets of $G$. In fact, this is equivalent to say that $(x_1, x_2)\in X^2$ is a Li-Yorke pair if and only if there exists a sequence $\{g_n: n\in \mathbb{N}\}$ of distinct elements of $G$ and a sequence $\{g_n': n\in \mathbb{N}\}$ of distinct elements of $G$ such that
$$\limsup_{n\rightarrow \infty} \rho (g_n x_1, g_n x_2)> 0\ \text{and}\ \liminf_{n\rightarrow \infty} \rho (g_n' x_1, g_n' x_2)= 0,$$
respectively. We say that the $G$-system $(X, G)$ is \emph{Li-Yorke chaotic} if there exists an uncountable subset $Z\subset X$ such that every non-diagonal pair $(x_1, x_2)\in Z^2$ is a Li-Yorke pair. Obviously, the definitions of Li-Yorke pairs and Li-Yorke chaos work only for an infinite group $G$. These definitions came from \cite{LiYorke}, and their relationship with entropy of a $G$-system was explored first for $\mathbb{Z}$-actions in \cite{BlanchardGlasnerKolyadaMaass} and then for actions of an infinite countable sofic group recently by Kerr and Li in \cite{KerrLisoficindependence}.

As shown by Theorem \ref{1305150007}, weakly mixing subsets also present chaotic behavior of actions even we consider dynamical systems of a countable discrete group.

\begin{thm} \label{1305150007}
Assume that $(X, G)$ admits a non-trivial weakly mixing subset of order 2. Then $G$ is infinite and $(X, G)$ is Li-Yorke chaotic.
\end{thm}

And then as a direct corollary of Theorem \ref{wmamenable} and Theorem \ref{1305150007}, we have:

\begin{cor} \label{1307182353}
$(X, G)$ is Li-Yorke chaotic if it has positive topological entropy, where $G$ is an infinite countable discrete amenable group.
\end{cor}

Note that, using Theorem \ref{1305150007} and Theorem \ref{1307151844}, Corollary \ref{1307182353} is in fact a special case of \cite[Corollary 8.4]{KerrLisoficindependence}, which is explained at the end of the paper.

As shown by the above results, all actions of a countable discrete group with positive topological entropy present very complicated limit behavior of points in the system. There are many other such evidences. Since its introduction in \cite{Schmidt95}, the homoclinic equivalence relation plays a very important role in the study of actions of a countable discrete group on a compact metric space by self-homeomorphisms as shown by \cite{BowenLi, ChungLi, EinsiedlerSchmidt, KatokSchmidt95, LindSchmidt, LindSchmidtVerbitskiy, Schmidt95, SchmidtVerbitskiy}, which was first explored for a shift of finite type of a $\mathbb{Z}^d$-action ($d\in \mathbb{N}$) by Schmidt \cite[Proposition 2.1]{Schmidt95}. Note that every expansive $\mathbb{Z}^d$-action by automorphisms of a compact zero-dimensional abelian group is topologically conjugate to a shift of finite type of a $\mathbb{Z}^d$-action \cite{KitchensSchmidt89}.
For an action of the integer group $\mathbb{Z}$, the relationship between topological entropy and asymptotic pairs (a variation of the homoclinic equivalence relation in the setting of a $\mathbb{Z}$-action) was observed by Blanchard, Host and Ruette in \cite{BlanchardHostRuette}.

The following result generalizes \cite[Proposition 2.1]{Schmidt95} to a shift of finite type for actions of an infinite countable residually finite amenable group (see \S \ref{homoclinic} for a detailed introduction of all related definitions).

\begin{thm} \label{1308020015}
Let $X\subset \{1, \cdots, m\}^G$ be a shift of finite type containing a dense set of periodic points, where $G$ is an infinite countable residually finite amenable group. Then $(X, G)$ has positive topological entropy if and only if its homoclinic equivalence relation $H_X$ is non-trivial.
\end{thm}

Recall that, if $G$ is a polycyclic-by-finite group acting on a compact abelian group expansively by automorphisms, the relationship between topological entropy and homoclinic points (another variation of the homoclinic equivalence relation in this setting) was discussed recently by Chung and Li in \cite{ChungLi}; and if $G$ is an infinite countable discrete amenable group with an algebraic past which includes $\mathbb{Z}^d$ for all $d\in \mathbb{N}$, the chaotic properties and asymptotic limit behaviors in a $G$-system with positive topological entropy were studied recently by Huang, Xu and Yi in \cite{HuangXuYi}. For a group $G$ with the unit $e_G$, by an \emph{algebraic past} $\mathcal{P}\subset G$ we mean that $\mathcal{P}$ satisfies $\mathcal{P}\cap \mathcal{P}^{- 1}= \emptyset, \mathcal{P}\cup \mathcal{P}^{- 1}\cup \{e_G\}= G$ and $\mathcal{P}\cdot \mathcal{P}\subset \mathcal{P}$.

The paper is organized as follows. In section 2 we study local weak mixing for actions of a general countable discrete group and then prove Theorem \ref{1305150007}. In section 3 we give some preliminaries of entropy theory for actions of an infinite countable discrete amenable group which will be used in later discussions. In section 4 after introducing the notion of a shift of finite type and homoclinic equivalence relation for an action of an infinite countable discrete amenable group, we prove
Theorem \ref{1308020015}. In section 5 we first prove Theorem \ref{wmamenable}, and then give some comments of the results by Kerr and Li in \cite{KerrLisoficindependence} related to the topic of the present paper.

\section{Weak mixing in countable discrete group actions}

In this section we study local weak mixing for actions of a countable discrete group. As a consequence, we prove that the existence of a non-trivial weakly mixing subset implies Li-Yorke chaos, c.f. Theorem \ref{1305150007} for details. We also prove that each weakly mixing subset in an equicontinuous action must be trivial.

Recall that $(X, G)$ is a $G$-system. We say that it is \emph{transitive} if $N (U, V)\neq \emptyset$ for any non-empty open subsets $U, V\subset X$, which is equivalent to say $\text{Tran} (X, G)\neq \emptyset$, where $\text{Tran} (X, G)$ denotes the set of all \emph{transitive points} in $(X, G)$, that is, $x\in \text{Tran} (X, G)$ if and only if \emph{the orbit of $x$}, $G x= \{g x: g\in G\}$, is dense in $X$.

The following observation is trivial. For completeness we present a proof here.

\begin{lem} \label{1307161017}
Let $(X, G)$ be a transitive $G$-system.
\begin{enumerate}

\item \label{1307171211} If $X$ is perfect, i.e., contains no isolated points, then for each $x\in \text{Tran} (X, G)$, $N (x, U)$ is an infinite set for any non-empty open subset $U\subset X$.

\item \label{1307171210} If $X$ is not perfect then each point of $\text{Tran} (X, G)$ is isolated in $X$.
\end{enumerate}
\end{lem}
\begin{proof}
(1) Let $U\subset X$ be a non-empty open subset. For each $F\in \mathcal{F}_G$, as $X$ contains no isolated points, $U\setminus \{f x: f\in F\}$ is also a non-empty open subset, and so there exists $g\in G$ with $g x\in U\setminus \{f x: f\in F\}$, which implies $N (x, U)\setminus F\neq \emptyset$.

(2) This item is also trivial. Let $x_0\in X$ be an isolated point and then let $x\in \text{Tran} (X, G)$. By the definition, it is easy to see $g x= x_0$ for some $g\in G$, equivalently, $x= g^{- 1} x_0$. And then the point $x$ is isolated in $X$.
\end{proof}

Similarly to \cite[Proposition 3.3]{PZstudia}, \cite[Proposition 2.3]{PZstudia}, \cite[Theorem 4.2]{PZETDS}, and \cite[Proposition 4.2]{BlanchardHuang} and \cite[Proposition 5.2]{PZETDS}, respectively, we have the following characterizations of weakly mixing subsets.

\begin{prop} \label{1305142253}
Let $\emptyset\neq K\subset X$ be a closed subset and $m\in \mathbb{N}$.

 \begin{enumerate}

 \item \label{1307171021} If $K$ is a non-trivial weakly mixing subset of order 2 then it is perfect.

 \item \label{1307171022} If $K$ is weakly mixing of order $m$, then $(X_{K, m}, G)$ is a transitive $G$-system, where $X_{K, m}$ is the smallest closed $G$-invariant subset of $X^m$ containing $K^m$.

 \item \label{1307171023} If $K$ is weakly mixing of order 1, then $(X_{K, 1}, G)$ is a transitive $G$-system, moreover, $K\cap \text{Tran} (X_{K, 1}, G)$ is a dense $G_\delta$ subset of $K$.

 \item \label{1307171024} If $K$ is non-trivial then $\eqref{1307140118}\Longleftrightarrow \eqref{1307140119}\Longrightarrow \eqref{1307140120}$, where
\begin{enumerate}

\item \label{1307140118} $K$ is weakly mixing of all orders;

\item \label{1307140119} there exists $B\subset K$, which is the union of countably many Cantor subsets, with properties that $B$ is a dense subset of $K$ and for each $E\subset B$ and any continuous mapping $h:E\rightarrow K$ there exists a sequence $\{g_n: n\in \mathbb{N}\}\subset G$ such that the sequence $\{g_n: n\in \mathbb{N}\}$ (of mappings over $X$) converges to $h$ pointwise on $E$;

\item \label{1307140120} for any given finite family of continuous mappings $f_i: K\rightarrow K, i= 1, \cdots, l$ with $l\in \mathbb{N}$, there exists an increasing sequence of Cantor subsets $C_1\subset C_2\subset \cdots\subset K$ with properties that the union of $\{C_n: n\in \mathbb{N}\}$ is dense in $K$ and for each $i= 1, \cdots, l$ there exists a sequence $\{g_n^{(i)}: n\in \mathbb{N}\}\subset G$ such that the sequence $\{g_n^{(i)}: n\in \mathbb{N}\}$ (of mappings over $X$) converges to $f_i$ uniformly on $C_s$ for any $s\in \mathbb{N}$.
\end{enumerate}
 \end{enumerate}
 \end{prop}

Recall that a well-known result by Furstenberg states that weak mixing of order 2 is equivalent to weak mixing of all orders for actions of an abelian group \cite{Furstenberg}. While, this is no longer valid for actions of non-abelian groups \cite{Glasner05, Weiss}.

By a \emph{Mycielski subset} we mean a subset which is the union of a sequence of Cantor subsets. The following version of the well-known Mycielski Theorem \cite[Theorem 1]{Mycielski} is very useful in the topological construction of Li-Yorke chaos.

\begin{lem} \label{1307161537}
Let $X$ be a perfect compact metric space and $R_n\subset X^{r_n}$ a subset of first category for each $n\in \mathbb{N}$. Then there exists a dense Mycielski set $M$ such that, for each $n\in \mathbb{N}$, $(x_i)_{i= 1}^{r_n}\notin R_n$ whenever $x_1, \cdots, x_{r_n}$ are $r_n$ distinct elements of $M$.
\end{lem}

In order to prove Theorem \ref{1305150007}, we also need the following useful observation.

\begin{prop} \label{1307161940}
Let $K$ be a non-trivial weakly mixing subset of order $n$ with $n\in \mathbb{N}\setminus \{1\}$. Assume that all $U_1, V_1, \cdots, U_n, V_n$ are open subsets of $X$ intersecting $K$.
Then
$N (U_1\cap K, V_1)\cap \cdots\cap N (U_n\cap K, V_n)$ is an infinite subset of $G$.
\end{prop}
\begin{proof}
Observe that, for $i= 1, \cdots, n, g\in G$ and $\overline{K}$, the closure of $K$,
$$g\in N (U_i\cap K, V_i)\Longleftrightarrow U_i\cap K\cap g^{- 1} V_i\neq \emptyset\Longleftrightarrow (U_i\cap \overline{K})\cap g^{- 1} V_i\neq \emptyset.$$
Thus we only need to show that $N (U_1\cap \overline{K}, V_1)\cap \cdots\cap N (U_n\cap \overline{K}, V_n)\subset G$ is infinite.

By the assumption one has that $\overline{K}$ is weakly mixing of order $n$, and then: $\overline{K}$ contains no isolated points by Proposition \ref{1305142253} (\ref{1307171021}); and
 the sub-$G$-system $(X_{\overline{K}, n}, G)$ of $(X^n, G)$ generated by $\overline{K}^n$ is a transitive $G$-system by Proposition \ref{1305142253} (\ref{1307171022}) and $\overline{K}^n\cap \text{Tran} (X_{\overline{K}, n}, G)$ is a dense $G_\delta$ subset of $\overline{K}^n$ by Proposition \ref{1305142253} (\ref{1307171023}). Say
 $$(x_1, \cdots, x_n)\in \overline{K}^n\cap \text{Tran} (X_{\overline{K}, n}, G)\cap \prod_{i= 1}^n U_i,$$
 such a tuple $(x_1, \cdots, x_n)$ must exist. As $\overline{K}$ is perfect, then $\overline{K}^n$ is also perfect, and hence
$\overline{K}^n\cap \text{Tran} (X_{\overline{K}, n}, G)$ contains no points which are isolated in $X_{\overline{K}, n}$, thus $X_{\overline{K}, n}$ is perfect by Lemma \ref{1307161017} (\ref{1307171210}). Moreover, applying Lemma \ref{1307161017} (\ref{1307171211}) one has that
$$\bigcap_{i= 1}^n N (U_i\cap \overline{K}, V_i)\supset \bigcap_{i= 1}^n N (x_i, V_i)= N \left((x_1, \cdots, x_n), \prod_{i= 1}^n V_i\right)$$
is an infinite set. This finishes the proof.
\end{proof}

With the help of Proposition \ref{1307161940}, we
can prove the following Proposition \ref{1305150017}, and then Theorem \ref{1305150007}
follows as a direct corollary of Proposition \ref{1305150017}.

\begin{prop} \label{1305150017}
Let $\emptyset\neq K\subset X$ be a non-trivial closed weakly mixing subset of order $n, n\in \mathbb{N}\setminus \{1\}$. Then there exists a dense Mycielski set $M\subset K$ such that, once $x_1, \cdots, x_n\in M$ are distinct and $x_1', \cdots, x_n'\in M$ (need not be distinct), then
\begin{equation} \label{1501110029}
\liminf_{G\ni g\rightarrow \infty} \max_{1\le i\le n} \rho (g x_i, x_i')= 0.
\end{equation}
\end{prop}
\begin{proof}
We shall construct a dense Mycielski subset $M\subset K$ such that \eqref{1501110029} holds when all points $x_1, x_1', \cdots, x_n, x_n'\in M$ are distinct. Note that $K$ is a non-trivial closed weakly mixing subset of order $n$, $K$ contains no isolated points by Proposition \ref{1305142253} (\ref{1307171021}), and hence its dense Mycielski subset $M$ contains no isolated points. From which it is easy to check that the constructed subset $M$ satisfies the conclusion.

In $\mathcal{F}_G$ we take a sequence $E_1\subset E_2\subset \cdots$ with union $G$. For each $k\in \mathbb{N}$ we put
$$S_k= \left\{(x_1, \cdots, x_n, x_1', \cdots, x_n')\in K^{2 n}: \inf_{g\in G\setminus E_k} \max_{1\le i\le n} \rho (g x_i, x_i')< \frac{1}{k}\right\}.$$
Obviously, $S_k\subset K^{2 n}$ is an open subset. In fact, $S_k\subset K^{2 n}$ is also dense, as: once $U_1, V_1, \cdots, U_n, V_n$ are non-empty open subsets of $X$ intersecting $K$, by shrinking we could assume additionally that the diameters of all subsets $V_1, \cdots, V_n$ are at most $\frac{1}{2 k}$, then
$$\bigcap_{i= 1}^n N (U_i\cap K, V_i)\setminus E_k\neq \emptyset$$
by Proposition \ref{1307161940}, and so we could choose
$$(x_1, \cdots, x_n, x_1', \cdots, x_n')\in S_k\cap \prod_{i= 1}^n (U_i\cap K)\times \prod_{i= 1}^n (V_i\cap K).$$
Thus $S$ is a dense $G_\delta$ subset of $K^{2 n}$, and for each $(x_1, \cdots, x_n, x_1', \cdots, x_n')\in S$,
$$\liminf_{G\ni g\rightarrow \infty} \max_{1\le i\le n} \rho (g x_i, x_i')= 0,\ \text{where}\ S= \bigcap_{k\in \mathbb{N}} S_k.$$
Now applying Lemma \ref{1307161537} to $K^{2 n}\setminus S$ we obtain a dense Mycielski subset $M\subset K$ such that $(x_1, \cdots, x_n, x_1', \cdots, x_n')\in S$ once all points $x_1, x_1', \cdots, x_n, x_n'\in M$ are distinct. This finishes the proof of the conclusion.
\end{proof}

Similarly, we have the following result, whose proof is left to the reader.

\begin{prop} \label{1307161935}
Let $\emptyset\neq K\subset X$ be a non-trivial closed weakly mixing set of all orders. Then there exists a dense Mycielski set $M\subset K$ such that, for each $n\in \mathbb{N}$, once $x_1, \cdots, x_n\in M$ are distinct and $x_1', \cdots, x_n'\in M$ (need not be distinct), then
$$\liminf_{G\ni g\rightarrow \infty} \max_{1\le i\le n} \rho (g x_i, x_i')= 0.$$
\end{prop}

We say that $(X, G)$ is \emph{equicontinuous} if for each $\epsilon> 0$ there exists $\delta> 0$ such that $d (x_1, x_2)< \delta$ implies $d (g x_1, g x_2)< \epsilon$ for all $g\in G$. In the following we discuss the local weak mixing in an equicontinuous system.

By a \emph{cover} of $X$ we mean a family of subsets of $X$ with union $X$. If elements of a cover are pairwise disjoint, then it is called a \emph{partition}.
Denote by $\mathcal{C}_X, \mathcal{C}_X^o$ and $\mathcal{P}_X$
the set of all finite Borel covers, finite open covers and finite Borel partitions of $X$, respectively.
Let $\mathcal{V}_1, \mathcal{V}_2\in \mathcal{C}_X$. We say that \emph{$\mathcal{V}_1$ is finer than $\mathcal{V}_2$} if each element of $\mathcal{V}_1$ is contained in some element of $\mathcal{V}_2$, which is denoted by $\mathcal{V}_1\succeq \mathcal{V}_2$.
For $F\in \mathcal{F}_G$ and $\emptyset\neq K\subset X$ we set $(\mathcal{V}_1)_F= \bigvee_{g\in F} g^{- 1} \mathcal{V}_1$, and set $N (\mathcal{V}_1, K)$ to be the minimal cardinality of sub-families of $\mathcal{V}_1$ covering $K$, we set $N (\mathcal{V}_1, \emptyset)= 0$ by convention.

Then similar to \cite[Proposition 2.2]{BlanchardHostMaass} and \cite[Proposition 3.7]{PZstudia} we have:

\begin{prop} \label{1305142353}
$(X, G)$ is equicontinuous if and only if for each $\mathcal{U}\in \mathcal{C}_X^o$ there exists $M< \infty$ such that $N (\mathcal{U}_F, X)< M$ for all $F\in \mathcal{F}_G$. In this case, once $\emptyset\neq K\subset X$ is a weakly mixing subset of order 2 then $K$ must be trivial.
\end{prop}

\section{Preliminaries for entropy of amenable group actions}

In this section we give some preliminaries of entropy theory for actions of an infinite countable discrete amenable group which will be useful in later discussions.

\emph{From now on, we will assume additionally that $G$ is an infinite countable discrete amenable group.}
Recall that a countable discrete group $G$ is \emph{amenable} if
there exists a sequence $\{F_n: n\in \mathbb{N}\}\subset \mathcal{F}_G$ such that
$$\lim_{n\rightarrow \infty} \frac{|g F_n\Delta F_n|}{|F_n|}= 0$$
for each $g\in G$, which is equivalent to mean that
$$\lim_{n\rightarrow \infty} \frac{|K F_n\cap K F_n^c|}{|F_n|}= 0$$
for each $K\in \mathcal{F}_G$, here we denote by $|\bullet|$ the cardinality of a set $\bullet$. Such a sequence $\{F_n: n\in \mathbb{N}\}$ is called a \emph{F\o lner sequence} of $G$.
In particular,
$$F_n= \prod_{i= 1}^d \{a_{n, i}, a_{n, i}+ 1, \cdots, a_{n, i}+ b_{n, i}- 1\}, n\in \mathbb{N}$$
will define a typical F\o lner sequence for amenable groups $\mathbb{Z}^d$ with $d\in \mathbb{N}$ when
$$\lim_{n\rightarrow \infty} \min_{i= 1}^d b_{n, i}= \infty.$$
The class of countable discrete amenable groups includes all solvable groups.

\emph{We will fix $\{F_n: n\in \mathbb{N}\}\subset \mathcal{F}_G$ to be a F\o lner sequence of the infinite countable discrete amenable group $G$ till the end of the paper.}

Let $\mathcal{U}\in \mathcal{C}_X^o$. Recall that the \emph{topological entropy of $\mathcal{U}$} is defined as
\begin{equation} \label{1307151659}
h_\text{top} (G, \mathcal{U})= \lim_{n\rightarrow \infty} \frac{1}{|F_n|} \log N (\mathcal{U}_{F_n}, X).
\end{equation}
As guaranteed by the well-known Ornstein-Weiss Lemma \cite[Theorem 6.1]{LW}, the limit is well defined and independent of the selection of a F\o lner sequence $\{F_n: n\in \mathbb{N}\}\subset \mathcal{F}_G$. Then the \emph{topological entropy of $(X, G)$} is defined as
\begin{equation} \label{1307172349}
h_\text{top} (G, X)= \sup_{\mathcal{V}\in \mathcal{C}_X^o} h_\text{top} (G, \mathcal{V}).
\end{equation}
Moreover, let $\emptyset\neq K\subset X$, we define the \emph{topological entropy of $K$} as
$$h_\text{top} (G, K)= \sup_{\mathcal{V}\in \mathcal{C}_X^o} \limsup_{n\rightarrow \infty} \frac{1}{|F_n|} \log N (\mathcal{V}_{F_n}, K).$$
And so if $K$ is a finite non-empty set then $h_\text{top} (G, K)= 0$.

Denote by $M (X)$ the set of all Borel probability measures on $X$; by $M (X, G)$ the set of all $G$-invariant elements $\mu$ in $M (X)$, i.e., $\mu (A)= \mu (g^{- 1} A)$ for each $g\in G$ and all $A\in \mathcal{B}_X$, where $\mathcal{B}_X$ is the Borel $\sigma$-algebra of $X$; and by $M^e (X, G)$ the set of all ergodic elements $\nu$ in $M (X, G)$, i.e., for $A\in \mathcal{B}_X$, $A= g A$ for all $g\in G$ implies $\nu (A)= 0$ or $1$.
 Note that $M^e (X, G)\neq \emptyset$ always holds for an amenable group $G$. For $\mu\in M (X)$, denote by $\text{supp} (\mu)$ the \emph{support of $\mu$}, the smallest closed non-empty subset $K\subset X$ with $\mu (K)= 1$. It is basic that if $\mu\in M (X, G)$ then $\text{supp} (\mu)$ is $G$-invariant and hence $(\text{supp} (\mu), G)$ is a $G$-system, and, additionally, if $\mu\in M^e (X, G)$ then the $G$-system $(\text{supp} (\mu), G)$ is transitive and in fact $\mu (\text{Tran} (\text{supp} (\mu), G))= 1$.

Let $\mathcal{A}\subset \mathcal{B}_X$ be a sub-$\sigma$-algebra and $\mu\in M (X), \alpha\in \mathcal{P}_X$. Set
$$H_\mu (\alpha| \mathcal{A})= - \sum_{A\in \alpha} \int_X \mathbb{E}_\mu (1_A| \mathcal{A}) \log \mathbb{E}_\mu (1_A| \mathcal{A}) d \mu,$$
where $\mathbb{E}_\mu (1_A| \mathcal{A})$ denotes the $\mu$-expectation of the function $1_A$ with respect to $\mathcal{A}$; and then for the trivial sub-$\sigma$-algebra $\mathcal{N}= \{\emptyset, X\}$, we set
$$H_\mu (\alpha)= H_\mu (\alpha| \mathcal{N})= - \sum_{A\in \alpha} \mu (A) \log \mu (A).$$
In general, for $\mathcal{U}\in \mathcal{C}_X$, we set without any ambiguity
$$H_\mu (\mathcal{U}| \mathcal{A})= \inf_{\alpha\in \mathcal{P}_X, \alpha\succeq \mathcal{U}} H_\mu (\alpha| \mathcal{A}).$$

Now let $\mu\in M (X, G), \mathcal{A}\subset \mathcal{B}_X$ be a $G$-invariant sub-$\sigma$-algebra (in the sense of $g \mathcal{A}= \mathcal{A}$ mod $\mu$ for each $g\in G$) and $\mathcal{U}\in \mathcal{C}_X$. Recalling from \cite{HYZJFA2011}, we may define the \emph{measure-theoretic $\mu$-entropy of $\mathcal{U}$ relative to $\mathcal{A}$} as
\begin{equation} \label{1307182302}
h_\mu (G, \mathcal{U}| \mathcal{A})= \lim_{n\rightarrow \infty} \frac{1}{|F_n|} H_\mu (\mathcal{U}_{F_n}| \mathcal{A})\ \left(= \inf_{\alpha\in \mathcal{P}_X, \alpha\succeq \mathcal{U}} h_\mu (G, \alpha| \mathcal{A})\right),
\end{equation}
and the \emph{measure-theoretic $\mu$-entropy of $\mathcal{U}$} as
\begin{equation*} \label{1307182303}
h_\mu (G, \mathcal{U})= h_\mu (G, \mathcal{U}| \mathcal{N})= \lim_{n\rightarrow \infty} \frac{1}{|F_n|} H_\mu (\mathcal{U}_{F_n})\ \left(= \inf_{\alpha\in \mathcal{P}_X, \alpha\succeq \mathcal{U}} h_\mu (G, \alpha)\right),
\end{equation*}
respectively, where the latter equalities follow from \cite[Theorem 3.2]{DooleyZhang} and \cite[Theorem 4.14]{HYZJFA2011}, respectively.
Similarly to \eqref{1307151659}, these limits are well defined and independent of the selection of a F\o lner sequence $\{F_n: n\in \mathbb{N}\}\subset \mathcal{F}_G$.
Then the \emph{measure-theoretic $\mu$-entropy of $(X, G)$} is defined as
\begin{equation} \label{1307172351}
h_\mu (G, X)= \sup_{\alpha\in \mathcal{P}_X} h_\mu (G, \alpha)\ \left(= \sup_{\mathcal{U}\in \mathcal{C}_X^o} h_\mu (G, \mathcal{U})\right),
\end{equation}
where the second identity follows from \cite[Theorem 3.5]{HYZJFA2011}.
For each $\mathcal{U}\in \mathcal{C}_X^o$,
one of the main results of \cite{HYZJFA2011}, \cite[Theorem 5.1]{HYZJFA2011}, shows that
$$h_\text{top} (G, \mathcal{U})= \max_{\mu\in M (X, G)} h_\mu (G, \mathcal{U})= \max_{\mu\in M^e (X, G)} h_\mu (G, \mathcal{U}),$$
and hence combined with \eqref{1307172349} and \eqref{1307172351} one has
\begin{equation} \label{1307172352}
h_\text{top} (G, X)= \sup_{\mu\in M (X, G)} h_\mu (G, X)= \sup_{\mu\in M^e (X, G)} h_\mu (G, X).
\end{equation}

For $\mu\in M (X, G)$ and a $G$-invariant sub-$\sigma$-algebra $\mathcal{A}\subset \mathcal{B}_X$, we define
$$P_\mu (G| \mathcal{A})= \{A\in \mathcal{B}_X: h_\mu (G, \{A, X\setminus A\}| \mathcal{A})= 0\}.$$
By the definition, $\mathcal{A}\cup \mathcal{I}_X\subset P_\mu (G| \mathcal{A})$, where $\mathcal{I}_X$ denotes the set of all $G$-invariant Borel subsets of $X$.
It is well known that $P_\mu (G| \mathcal{A})= P_\mu (G| P_\mu (G| \mathcal{A}))$ and $P_\mu (G| \mathcal{A})\subset \mathcal{B}_X$ is a $G$-invariant sub-$\sigma$-algebra containing $\mathcal{A}$, which is called
the \emph{Pinsker $\sigma$-algebra of $\mu$ relative to $\mathcal{A}$}. In particular, we set $P_\mu (G)= P_\mu (G| \mathcal{N})$ (and hence $P_\mu (G)= P_\mu (G| P_\mu (G))$) and call it the \emph{Pinsker $\sigma$-algebra of $\mu$}.
Moreover, for $\alpha\in \mathcal{P}_X$ and $\mathcal{U}\in \mathcal{C}_X$, one has (observing \eqref{1307182302})
\begin{equation} \label{1307182308}
h_\mu (G, \alpha)= h_\mu (G, \alpha| P_\mu (G))\ \text{and hence}\ h_\mu (G, \mathcal{U})= h_\mu (G, \mathcal{U}| P_\mu (G)).
\end{equation}
  Now for each $n\in \mathbb{N}\setminus \{1\}$, following
ideas from \cite{Glasnerisrael, HYZJFA2011}, we could introduce a probability measure $\lambda_{n, \mu}$ over $(X^n, \mathcal{B}_X^n)$ by, for all $A_1, \cdots, A_n\in \mathcal{B}_X$, setting
$$\lambda_{n, \mu} \left(\prod_{i= 1}^n A_i\right)= \int_X \prod_{i= 1}^n \mathbb{E}_\mu (1_{A_i}| P_\mu (G)) d \mu.$$
In fact, observing the $G$-invariance of $P_\mu (G)$ one has $\lambda_{n, \mu}\in M (X^n, G)$.

Then we have the following observation:

\begin{prop} \label{1307172302}
Let $\mu\in M^e (X, G)$ and $n\in \mathbb{N}\setminus \{1\}$. Then
\begin{enumerate}

\item $h_{\lambda_{n, \mu}} (G, X^n)= n h_\mu (G, X)$.

\item $\lambda_{n, \mu}\in M^e (X^n, G)$ (and hence $(\text{supp} (\lambda_{n, \mu}), G)$ is a transitive $G$-system).

\item If $h_\mu (G, X)> 0$ then $\lambda_{n, \mu} (\Delta_n)= 0$, where
$\Delta_n= \{(x_1, \cdots, x_n): x_1= \cdots= x_n\in X\}$.
 Moreover, if let $\mu= \int_X \mu_x d \mu (x)$ be the disintegration of $\mu$ over its Pinsker $\sigma$-algebra $P_\mu (G)$, then: for $\mu$-a.e. $x\in X$, $\mu_x$ is non-atomic (and hence $\text{supp} (\mu_x)$ contains no isolated points) and $\text{supp} (\mu_x)^n\cap \text{Tran} (\text{supp} (\lambda_{n, \mu}), G)$ is a dense subset of $\text{supp} (\mu_x)^n$.
\end{enumerate}
\end{prop}
\begin{proof}
(1) The proof is similar to that of
\cite[Lemma 3.1]{DouYeZhang}.

(2) The proof is similar to that of \cite[Lemma 4.3]{HuangXuYi}.
Let $\pi_n: X^n\rightarrow X$ be the projection $(x_1, \cdots, x_n)\mapsto x_1$. By applying $n- 1$ times of \cite[Theorem 4]{GlasnerThouvenotWeiss} we obtain $P_{\lambda_{n, \mu}} (G| \pi_n^{- 1} P_\mu (G))= \pi_n^{- 1} P_\mu (G)$ (in the sense of mod $\lambda_{n, \mu}$), that is, apply \cite[Theorem 4]{GlasnerThouvenotWeiss} to obtain it for $\lambda_{2, \mu}$ and then apply \cite[Theorem 4]{GlasnerThouvenotWeiss} to obtain it for $\lambda_{3, \mu}$, and so on. Now let $A_n\subset X^n$ be any $G$-invariant Borel subset. Then $A_n\in P_{\lambda_{n, \mu}} (G| \pi_n^{- 1} P_\mu (G))$, and hence $A_n= \pi_n^{- 1} (C)$ mod $\lambda_{n, \mu}$ for some $C\in \mathcal{B}_X$. Then, by combining with the $G$-invariance of $A_n$ and $\pi_n \lambda_{n, \mu}= \mu$, we obtain
\begin{eqnarray*}
\mu (g C\Delta C)&= & \lambda_{n, \mu} (\pi_n^{- 1} (g C\Delta C))= \lambda_{n, \mu} (\pi_n^{- 1} (g C)\Delta \pi_n^{- 1} C) \\
&= & \lambda_{n, \mu} (g \pi_n^{- 1} (C)\Delta \pi_n^{- 1} C)= \lambda_{n, \mu} (g A_n\Delta A_n)= 0
\end{eqnarray*}
for each $g\in G$, and hence $\lambda_{n, \mu} (A_n)= \lambda_{n, \mu} (\pi_n^{- 1} C)= \mu (C)$ takes value of either 0 or 1 by the ergodicity of $\mu$. In particular, $\lambda_{n, \mu}\in M^e (X^n, G)$.

(3)
Now assume $h_\mu (G, X)> 0$. First we prove $\lambda_{n, \mu} (\Delta_n)= 0$. Else, by the $G$-invariance of $\Delta_n$ one has $\lambda_{n, \mu} (\Delta_n)= 1$, and hence, for each $A\in \mathcal{B}_X$ and set $A_1= A, A_2= X\setminus A$ and $A_3= \cdots= A_n= X$, one has
\begin{eqnarray*}
0= \lambda_{n, \mu} \left(\prod_{i= 1}^n A_i\right)&= & \int_X \prod_{i= 1}^n \mathbb{E}_\mu (1_{A_i}| P_\mu (G)) d \mu \\
&= & \int_X \mathbb{E}_\mu (1_A| P_\mu (G))\cdot \mathbb{E}_\mu (1_{X\setminus A}| P_\mu (G)) d \mu \\
&= & \int_X \mathbb{E}_\mu (1_A| P_\mu (G))\cdot [1- \mathbb{E}_\mu (1_A| P_\mu (G))] d \mu,
\end{eqnarray*}
and hence $\mathbb{E}_\mu (1_A| P_\mu (G))= 1_A$ mod $\mu$, equivalently, $A\in P_\mu (G)$. Thus, $P_\mu (G)= \mathcal{B}_X$, a contradiction to the assumption $h_\mu (G, X)> 0$.

Observing that by the construction and the assumption, we have
$$\lambda_{n, \mu}= \int_X \mu_x^n d \mu (x)\ \text{and hence}\ \lambda_{n, \mu} (\Delta_n)= \int_X (\mu_x^n) (\Delta_n) d \mu (x),$$
which implies $(\mu_x^n) (\Delta_n)= 0$ and then $\mu_x$ is non-atomic for $\mu$-a.e. $x\in X$. Moreover, by the conclusion of $\lambda_{n, \mu}\in M^e (X^n, G)$ we obtain
$$1= \lambda_{n, \mu} (\text{Tran} (\text{supp} (\lambda_{n, \mu}), G))= \int_X (\mu_x^n) (\text{Tran} (\text{supp} (\lambda_{n, \mu}), G)) d \mu (x),$$
and so, for $\mu$-a.e. $x\in X$, $(\mu_x^n) (\text{Tran} (\text{supp} (\lambda_{n, \mu}), G))= 1$, which implies the density of $\text{supp} (\mu_x)^n\cap \text{Tran} (\text{supp} (\lambda_{n, \mu}), G)$ in $\text{supp} (\mu_x)^n$, as $\text{supp} (\mu_x^n)= \text{supp} (\mu_x)^n$.
\end{proof}

\section{Homoclinic equivalence relation in a shift of finite type\\ for amenable group actions with positive entropy} \label{homoclinic}

Asymptotic limit behaviors in an action with positive topological entropy was observed for a $\mathbb{Z}$-action in \cite{BlanchardHostRuette}, for a shift of finite type of $\mathbb{Z}^d$-actions ($d\in \mathbb{N}$) in \cite{Schmidt95} and for an algebraic action in a series of papers \cite{BowenLi, ChungLi, EinsiedlerSchmidt, KatokSchmidt95, LindSchmidt, LindSchmidtVerbitskiy, SchmidtVerbitskiy}, respectively. In this section we study asymptotic limit behaviors of a shift of finite type with positive topological entropy for a general infinite countable discrete amenable group. The main result of this section is Theorem \ref{1308020015}, which is a direct corollary of Proposition \ref{1106052208} and Proposition \ref{1307262129}. Note that the proofs of Proposition \ref{1106052208} and Proposition \ref{1307262129} are inspired by that of \cite[Proposition 2.1]{Schmidt95}.

First let us introduce a shift of finite type for actions of a general amenable group.
Let $m\in \N\setminus \{1\}$ and fix it in this section. Then $(\{1, \cdots, m\}^G, G)$ is a $G$-system, where $G$ acts naturally over the compact metric space $\{1, \cdots, m\}^G$ by
$$g' (x_g: g\in G)= (x_{g g'}: g\in G).$$
By a \emph{sub-shift} we mean a closed $G$-invariant non-empty subset $X\subset \{1, \cdots, m\}^G$, i.e., $g X= X$ for each $g\in G$. In this case, we will also say that $(X, G)$ is a \emph{sub-shift}.
 For each $F\in \mathcal{F}_G$, denote by $\pi_F: \{1, \cdots, m\}^G\rightarrow \{1, \cdots, m\}^F$ the natural projection.
Then it is easy to check that the topological entropy of a sub-shift $(X, G)$ can be defined equivalently as
\begin{equation} \label{1307151215}
h_\text{top} (G, X)= \lim_{n\rightarrow \infty} \frac{1}{|F_n|} \log |\pi_{F_n} (X)|.
\end{equation}

A sub-shift $X$ is called a \emph{shift of finite type} if
$$X= \{(x_g)_{g\in G}: \pi_F\circ g' (x_g: g\in G)\in X_F\ \text{for each}\ g'\in G\}$$
for some $X_F\subset \{1, \cdots, m\}^F$ with $F\in \mathcal{F}_G$,
equivalently,
$$X= \{(x_g)_{g\in G}: (x_{g g'})_{g\in F}\in X_F\ \text{for each}\ g'\in G\}.$$
As shown by \cite[Proposition 3.12 and Theorem 5.2]{KitchensSchmidt89} (see also \cite[Theorem 3.2]{KatokSchmidt95}), each expansive $\Z^d$-system with $d\in \mathbb{N}$ by continuous automorphisms of a compact zero-dimensional abelian group is topologically conjugate to the shift-action of $\Z^d$ on a shift of finite type.
Here, we mean $(X, G)$ \emph{expansive} if there exists $\delta> 0$ such that, for all different points $x_1$ and $x_2$ in $X$, $\rho (g x_1, g x_2)> \delta$ for some $g\in G$.

The \emph{homoclinic equivalence relation} $H_X$ of a sub-shift $X$ is defined by, for $x= (x_g: g\in G)\in X$ and $x'= (x_g': g\in G)\in X$, $(x, x')\in H_X$ if and only if there exists $F\in \mathcal{F}_G$ such that $x_g= x_g'$ for all $g\in G\setminus F$, equivalently, for each $\epsilon> 0$ there exists $E\in \mathcal{F}_G$ with the property of $\rho (h x, h x')< \epsilon$ for each $h\in G\setminus E$.

For a shift-action of $\Z^d$ on a shift of finite type $X$,
\cite[Proposition 2.1]{Schmidt95} shows that $H_X$ must be \emph{non-trivial} in the sense of $(x, x')\in H_X$ for some $x\neq x'$ if the sub-shift $X$ has positive topological entropy. The following result shows that in fact it holds for a sub-shift of actions of a general amenable group.

\begin{prop} \label{1106052208}
Let $X\subset \{1, \cdots, m\}^G$ be a shift of finite type with positive topological entropy. Then $H_X$ is non-trivial.
\end{prop}
\begin{proof}
By the assumption, we let $X_F\subset \{1, \cdots, m\}^F$ with $F\in \mathcal{F}_G$ such that
$$X= \{(x_g)_{g\in G}: (x_{g g'})_{g\in F}\in X_F\ \text{for each}\ g'\in G\}.$$

Assume the contrary that $H_X$ is trivial, i.e., for $x, x'\in X$, $(x, x')\in H_X$ if and only if $x= x'$.
Then for all $x, x'\in X$ and each $E\in \mathcal{F}_G$, say $x= (x_g: g\in G)$ and $x'= (x_g': g\in G)$, once $x_g= x'_g$ for all $g\in E\cap F F^{- 1} E^c$, we have that $x$ agrees with $x'$ over $E$ (which will be proved later), which implies that $\pi_E (X)= \pi_{E\cap F F^{- 1} E^c} (X)$ for each $E\in \mathcal{F}_G$. Thus, according to \eqref{1307151215}, we have
\begin{eqnarray*}
h_\text{top} (G, X)&= & \lim_{n\rightarrow \infty} \frac{1}{|F_n|} \log |\pi_{F_n} (X)|= \lim_{n\rightarrow \infty} \frac{1}{|F_n|} \log |\pi_{F_n\cap F F^{- 1} F_n^c} (X)| \\
&\le & \log m\cdot \limsup_{n\rightarrow \infty} \frac{1}{|F_n|} |F_n\cap F F^{- 1} F_n^c|= 0,
\end{eqnarray*}
a contradiction to the assumption that $(X, G)$ has positive topological entropy.

 Thus it remains to show that, for any given $x, x'\in X$ and $E_1\in \mathcal{F}_G$ (fix them), say $x= (x_g: g\in G)$ and $x'= (x_g': g\in G)$, if $x_g= x'_g$ for all $g\in E_1\cap F F^{- 1} E_1^c$, then $x$ agrees with $x'$ over the whole $E_1$.
Assume the contrary that $x$ agrees with $x'$ over $E_1\cap F F^{- 1} E_1^c$ but not over $E_1$. We may choose $y\in \{1, \cdots, m\}^G\setminus \{x\}$ with $y_g= x'_g$ for each $g\in E_1$ and $y_g= x_g$ for each $g\in E_1^c$.
 Recalling $x, x'\in X$, one has:
  \begin{enumerate}

  \item For any $g'\notin F^{- 1} E_1^c$, $F g'\subset E_1$ and so $(y_{g g'})_{g\in F}= (x'_{g g'})_{g\in F}\in X_F$.

  \item For any $g'\in F^{- 1} E_1^c$ and let $g\in F$: if $g g'\in E_1$ then $g g'\in E_1\cap F F^{- 1} E_1^c$ and so $y_{g g'}= x'_{g g'}= x_{g g'}$ (observe that $x$ agrees with $x'$ over $E_1\cap F F^{- 1} E_1^c$), and if $g g'\notin E_1$ then $y_{g g'}= x_{g g'}$. Summing up, $(y_{g g'})_{g\in F}= (x_{g g'})_{g\in F}\in X_F$ for each $g'\in F^{- 1} E_1^c$.
  \end{enumerate}
This implies $y\in X$, and so $(x, y)\in H_X$ by the construction, a contradiction to the assumption that $H_X$ is trivial.
This finishes the proof.
\end{proof}

The converse of Proposition \ref{1106052208} holds true when the action is good enough and the infinite amenable group $G$ is residually finite as shown by Proposition \ref{1307262129}.

Recall that a countable group is \emph{residually finite} if the intersection of all its normal subgroups of finite indexes is the trivial group $\{e_G\}$, where $e_G$ is the unit of the group $G$. Each finitely generated nilpotent group is residually finite. Let $x\in X$. It is easy to see that $G_x\doteq \{g\in X: g x= x\}$ is a subgroup of $G$. We say that $x\in X$ is \emph{periodic} if $G x\subset X$ is a finite subset, which is equivalent to say that the subgroup $G_x$ has a finite index with respect to $G$.

As shown by \cite[Theorem 1.3]{BowenLi}, \cite[Theorem 5.7]{DeningerSchmidtetds}, \cite[Theorem 1.1]{LindSchmidtVerbitskiy} and \cite[Theorem 7.1]{LindSchmidtWard}, the concept of periodic points is quite related to the topological entropy for many algebraic actions of a residually finite group. And the connection between periodic points, topological entropy and homoclinic equivalence relation for a shift-action of $\Z^d$ on a shift of finite type was explored in \cite{Schmidt95}. Observe that, even only considering a shift-action of $\Z^d$ for a general $d> 1$, a shift of finite type may not contain any periodic point \cite{Ber, Rob}, and this potential absence of periodic points is associated with some of the difficulties one encounters when dealing with shifts of finite type in higher dimensions as shown in \cite{Schmidt95}.

Now we can provide the converse of Proposition \ref{1106052208} in some sense as follows.

\begin{prop} \label{1307262129}
Let $X\subset \{1, \cdots, m\}^G$ be a shift of finite type with zero topological entropy, where $G$ is an infinite countable residually finite amenable group. Assume that $X$ contains a dense set of periodic points. Then $H_X$ is trivial.
\end{prop}
\begin{proof}
By the assumption, we let $X_F\subset \{1, \cdots, m\}^F$ with $F\in \mathcal{F}_G$ such that
$$X= \{(x_g)_{g\in G}: (x_{g g'})_{g\in F}\in X_F\ \text{for each}\ g'\in G\}.$$
Now we assume that $H_X$ is non-trivial.
In the following we shall prove the conclusion by showing that in this case $(X, G)$ has positive topological entropy.

As $G$ is an infinite countable residually finite amenable group, by applying \cite[Lemma 5]{CortezPetiteggd} which is due to B. Weiss (see also \cite[Corollary 5.6]{DeningerSchmidtetds} or \cite[Theorem 1]{W0}), we could choose a sequence $\{G_n: n\in \mathbb{N}\}$ of normal subgroups of $G$ with finite indexes
and
 a F\o lner sequence $\{F_n: n\in \mathbb{N}\}$ of $G$ such that
 \begin{enumerate}

 \item
 $\{G_n: n\in \mathbb{N}\}$ decreases to the trivial group $\{e_G\}$,

 \item
$\{F_n g: g\in G_n\}$ forms a partition of $G$ for each $n\in \mathbb{N}$, and

 \item the sequence $e_G\in F_1\subset F_2\subset \cdots$ increases to the whole group $G$.
 \end{enumerate}
Observe that $H_X$ is non-trivial,
in $X$ we could choose different points $x= (x_g: g\in G)$ and $x'= (x_g': g\in G)$ and $N\in \mathbb{N}$ such that $x_g= x_g'$ for all $g\in G\setminus F_N$ (and hence $\pi_{F_N} (x)\neq \pi_{F_N} (x')$). As the sequence $F_1\subset F_2\subset \cdots$ increases to $G$, there exists $M\ge N$ with $F F^{- 1} F_N\subset F_M$ (and hence $\pi_{F_M} (x)\neq \pi_{F_M} (x')$). And finally, because that $X$ contains a dense set of periodic points, we could select a periodic point $y= (y_g: g\in G)\in X$ with $\pi_{F_M} (y)= \pi_{F_M} (x)$. Observe that $Q\doteq \{g\in G: g y= y\}$ is a subgroup of $G$ with a finite index (as $y$ is a periodic point).

As $\{F_M g: g\in G_M\}$ forms a partition of $G$, for each $z= (z_s: s\in Q\cap G_M)\in \{0, 1\}^{Q\cap G_M}$, there exists uniquely a point $x^z= (x^z_g: g\in G)\in \{1, \cdots, m\}^G$ with
\begin{eqnarray*}
\pi_{F_M} (s x^z)= & \pi_{F_M} (x)&\ \text{for each}\ s\in Q\cap G_M\ \text{with}\ z_s= 0,\\
\pi_{F_M} (s x^z)= & \pi_{F_M} (x')&\ \text{for each}\ s\in Q\cap G_M\ \text{with}\ z_s= 1,\\
x^z_g= & y_g&\ \text{for each}\ g\in G\setminus F_M (Q\cap G_M).
\end{eqnarray*}
In fact, $x^z_g= y_g$ for each $g\in G\setminus F_N (Q\cap G_M)$: by the above construction we only need to check it for each $g\in F_M (Q\cap G_M)\setminus F_N (Q\cap G_M)$, say $g= l r$ with $r\in Q\cap G_M$ and $l\in F_M$ (and hence $l\notin F_N$),
$$x^z_g= x^z_{l r}= (r x^z)_l= x_l\ \text{or}\ x'_l,$$
depending on $z_r= 0$ or 1, thus (observing $x'_l= x_l= y_l$ by $l\in F_M\setminus F_N$)
 $$x^z_g= y_l= (r y)_l\ (\text{as $r\in Q$})= y_{l r}= y_g.$$
Thus the constructed point $x^z$ belongs to the shift $X$:
\begin{itemize}

\item if $g'\in F^{- 1} F_N (Q\cap G_M)$, say $g'= f t$ with $t\in Q\cap G_M$ and $f\in F^{- 1} F_N$, then $F f\subset F_M$ by the construction and hence
     \begin{equation*}
     (x^z_{g g'})_{g\in F}= (x^z_{g f t})_{g\in F}= ((t x^z)_{g f})_{g\in F}= (x_{g f})_{g\in F}\ \text{or}\ (x'_{g f})_{g\in F},
     \end{equation*}
     depending on $z_t= 0$ or 1, which implies $(x^z_{g g'})_{g\in F}\in X_F$ because $x, x'\in X$;

\item if $g'\notin F^{- 1} F_N (Q\cap G_M)$ then $g g'\notin F_N (Q\cap G_M)$ for each $g\in F$ and so $(x^z_{g g'})_{g\in F}= (y_{g g'})_{g\in F}\in X_F$
by the above constructions (observing $y\in X$).
\end{itemize}

Now we aim to estimate the topological entropy of $(X, G)$ as follows.

Firstly, observe that $G_M$ is a normal subgroup of $G$ with a finite index and $Q$ is a subgroup of $G$ with a finite index. It is basic to see that $Q\cap G_M$ is a normal subgroup of $Q$, $G_M$ is a normal subgroup of $G_M Q$ (observing $G_M Q$ is a subgroup of $G$) and the quotient groups $(G_M Q)/G_M$ and $Q/ (Q\cap G_M)$ are isomorphic. Thus $Q\cap G_M$ is a subgroup of $Q$ with a finite index, and hence a subgroup of $G$ with a finite index, which implies that $F_L (Q\cap G_M)= G$ for some $L\ge M$.

Now fix each $K\ge M$. On one hand, for any given $z= (z_s: s\in Q\cap G_M)\in \{0, 1\}^{Q\cap G_M}$ and $z'= (z_s': s\in Q\cap G_M)\in \{0, 1\}^{Q\cap G_M}$, if $z_\xi\neq z'_\xi$ for some $\xi\in Q\cap G_M$ with $F_L \xi\subset F_K$, say $z_\xi= 0$ and $z'_\xi= 1$, then
$$\pi_{F_M} (\xi x^z)= \pi_{F_M} (x)\neq \pi_{F_M} (x')= \pi_{F_M} (\xi x^{z'})\ (\text{observing $\pi_{F_M} (x)\neq \pi_{F_M} (x')$}),$$
which implies $\pi_{F_K} (\xi x^z)\neq \pi_{F_K} (\xi x^{z'})$ (observing $K\ge M$), that is, $\pi_{F_K} (\xi x^z)$ and $\pi_{F_K} (\xi x^{z'})$ are two different elements of $\pi_{F_K} (X)$. In particular,
\begin{equation} \label{1312130000}
|\pi_{F_K} (X)|\ge 2^{|\{g\in Q\cap G_M: F_L g\subset F_K\}|}.
\end{equation}
On the other hand, as $F_L (Q\cap G_M)= G$, it is easy to see
\begin{equation} \label{1307312107}
|\{g\in Q\cap G_M: F_L g\cap F_K\neq \emptyset\}|\ge \frac{|F_K|}{|F_L|},
\end{equation}
and hence
\begin{eqnarray} \label{1307312019}
|\{g\in Q\cap G_M: F_L g\subset F_K\}|&\ge & |\{g\in Q\cap G_M: F_L g\cap F_K\neq \emptyset\}|-\nonumber \\
& & |\{g\in G: F_L g\cap F_K\neq \emptyset\ \text{and}\ F_L g\cap F_K^c\neq \emptyset\}|\nonumber \\
&\ge & \frac{|F_K|}{|F_L|}- |F_L^{- 1} F_K\cap F_L^{- 1} F_K^c|\ (\text{applying \eqref{1307312107}}).
\end{eqnarray}

Combining \eqref{1312130000} and \eqref{1307312019} with \eqref{1307151215} we obtain
\begin{eqnarray*}
h_\text{top} (G, X)&= & \lim_{K\rightarrow \infty} \frac{1}{|F_K|} \log |\pi_{F_K} (X)| \\
&\ge & \log 2\cdot \liminf_{K\rightarrow \infty} \frac{1}{|F_K|} |\{g\in Q\cap G_M: F_L g\subset F_K\}|\ge \frac{\log 2}{|F_L|}> 0
\end{eqnarray*}
which finishes the proof,
where the last inequality follows from the estimation \eqref{1307312019} and the fact that $\{F_n: n\in \mathbb{N}\}$ is a F\o lner sequence of $G$.
\end{proof}

\section{Weak mixing in amenable group actions with positive entropy}

The relationship between positive topological entropy and (local) weak mixing has been explored in \cite{BlanchardHuang, PZETDS, PZstudia, PZtopoappli} for a $\Z$-system. It is natural to ask if these results can be generalized to a $G$-system for a general amenable group. The answers turn out to be affirmative.
For example, \cite[Theorem 7.8]{HYZJFA2011} tells us that $X$ is weakly mixing of all orders if $(X, G)$ has uniformly positive entropy (see \cite{HYZJFA2011} for the detailed definition of uniform positive entropy for actions of an infinite countable discrete amenable group).
In this section first we aim to prove Theorem \ref{wmamenable} that (local) weak mixing exists indeed in such an action with positive topological entropy.

Before proceeding,
we need the following basic facts.

\begin{lem} \label{1307181004}
Let $(X, G)$ be a transitive $G$-system and $\emptyset\neq K\subset X$. Assume that $K\cap \text{Tran} (X, G)$ is a dense subset of $K$. Then $K$ is weakly mixing of order 1.
\end{lem}

\begin{lem} \label{1307182256}
Let $\mu\in M (X)$ and $\mathcal{U}\in \mathcal{C}_X$. Then $H_\mu (\mathcal{U})\le \log N (\mathcal{U}, \text{supp} (\mu))$.
\end{lem}

Then Theorem \ref{wmamenable} comes from \eqref{1307172352} and the following result.

\begin{prop} \label{1307181026}
Let $\mu\in M^e (X, G)$ with $h_\mu (G, X)> 0$, and say $\mu= \int_X \mu_x d \mu (x)$ to be the disintegration of $\mu$ over its Pinsker $\sigma$-algebra $P_\mu (G)$.
Then $\text{supp} (\mu_x)$ is a non-trivial weakly mixing subset of all orders for $\mu$-a.e. $x\in X$, and
\begin{equation} \label{1307182208}
h_\mu (G, X)\le \text{ess}-\sup_\mu\ h_\text{top} (G, \text{supp} (\mu_x))\ \text{with}\ \text{ess}-\sup_\mu= \inf_{A\in \mathcal{B}_X, \mu (A)= 1} \sup_{x\in A}.
\end{equation}
\end{prop}
\begin{proof}
Applying Proposition \ref{1307172302} we know that:
for $\mu$-a.e. $x\in X$, $\mu_x$ is non-atomic (and hence $\text{supp} (\mu_x)$ is non-trivial), and, for each $n\in \mathbb{N}$, $\text{supp} (\mu_x)^n\cap \text{Tran} (\text{supp} (\lambda_{n, \mu}), G)$ is a dense subset of $\text{supp} (\mu_x)^n$ (in particular, $\text{supp} (\mu_x)^n\subset \text{supp} (\lambda_{n, \mu})$) and hence $\text{supp} (\mu_x)^n$ is weakly mixing of order 1 by Lemma \ref{1307181004}, which is equivalent to say that $\text{supp} (\mu_x)$ is weakly mixing of order $n$. This implies that $\text{supp} (\mu_x)$ is a non-trivial weakly mixing subset of all orders for $\mu$-a.e. $x\in X$.

Now let $\mathcal{U}\in \mathcal{C}_X^o$, first we aim to prove
\begin{equation} \label{1307182228}
h_\mu (G, \mathcal{U})\le \text{ess}-\sup_\mu\ \limsup_{n\rightarrow \infty} \frac{1}{|F_n|} \log N (\mathcal{U}_{F_n}, \text{supp} (\mu_x)).
\end{equation}
In fact, with the help of Lemma \ref{1307182256}, it is easy to obtain \eqref{1307182228} as follows:
\begin{eqnarray} \label{1307182328}
h_\mu (G, \mathcal{U})&= & h_\mu (G, \mathcal{U}| P_\mu (G))\ (\text{using \eqref{1307182308}})\nonumber \\
&= & \lim_{n\rightarrow \infty} \frac{1}{|F_n|} H_\mu (\mathcal{U}_{F_n}| P_\mu (G))=  \lim_{n\rightarrow \infty} \frac{1}{|F_n|} \int_X H_{\mu_x} (\mathcal{U}_{F_n}) d \mu (x) \\
&\le & \int_X \limsup_{n\rightarrow \infty} \frac{1}{|F_n|} H_{\mu_x} (\mathcal{U}_{F_n}) d \mu (x)\le \text{ess}-\sup_\mu\ \limsup_{n\rightarrow \infty} \frac{1}{|F_n|} H_{\mu_x} (\mathcal{U}_{F_n})\nonumber \\
&\le & \text{ess}-\sup_\mu\ \limsup_{n\rightarrow \infty} \frac{1}{|F_n|} \log N (\mathcal{U}_{F_n}, \text{supp} (\mu_x)),\nonumber
\end{eqnarray}
where the second identity of \eqref{1307182328} is well known (see for example \cite[Lemma 3.13]{DooleyZhang}).
Now we take a sequence $\{\mathcal{U}_n: n\in \mathbb{N}\}\subset \mathcal{C}_X^o$ with the diameters of $\mathcal{U}_n$ tending to 0.
Then with help of \eqref{1307172351}, \eqref{1307182208} follows by firstly applying \eqref{1307182228} to each $\mathcal{U}_n$ and then taking the limit as $n$ tends to infinity. This finishes the proof.
\end{proof}

As shown by \cite{HYZJFA2011} and the above arguments, the proofs of \cite[Theorem 7.8]{HYZJFA2011} and Theorem \ref{wmamenable} both use ideas from ergodic theory. Recently, using mainly combinatorial
and topological methods, in \cite{KerrLisoficindependence} Kerr and Li undertook a local analysis of combinatorial independence which connects to topological entropy within the framework of actions of sofic groups. At the end of this section, we shall give some comments about its relationship with the topic of the present paper using the languages of \cite{KerrLisoficindependence}. See \cite{KerrLisoficindependence} for a detailed introduction of related discussions.


When considering a $G$-system $(X, G)$ where $G$ is an infinite countable discrete sofic group, Kerr and Li proved that positive (sofic) topological entropy also implies Li-Yorke chaos \cite[Corollary 8.4]{KerrLisoficindependence}, and then the system $(X, G)$ has at most zero (sofic) topological entropy if the space $X$ contains at most countably many points.
Note that \cite[Corollary 8.4]{KerrLisoficindependence} (and \cite[Theorem 8.1]{KerrLisoficindependence}) are stated for the sofic topological entropy defined using ultrafilter. However, the limsup versions of them follow directly from the ultrafilter versions, since one can pass to a subsequence where the quantity converges and then one can use any free ultrafilter on this subsequence.

Due to lack of space, the definition of sofic group is omitted here. See \cite{KerrLisoficindependence} for the detailed definition of soficity for a group and the introduction of entropy theory for actions of countable discrete sofic groups (see also \cite{BowenJAMS, KerrLiinvention} for the entropy theory for actions of countable discrete sofic groups).

Generalizations of \cite[Theorem 7.8]{HYZJFA2011} and Theorem \ref{wmamenable} to actions of a more general infinite countable discrete sofic group were proved in fact implicitly in \cite{KerrLisoficindependence}.


\begin{prop} \label{1307151842}
Assume that $G$ is an infinite countable discrete sofic group and $(X, G)$ has uniformly positive entropy (with respect to some fixed sofic approximation sequence $\Sigma$), i.e., $IE^\text{sof}_2 (X, G)= IE_2^\Sigma (X, G)= X^2$ in the language of \cite[Definition 4.3 and Remark 4.4]{KerrLisoficindependence}. Then $X$ is a weakly mixing set of all orders.
\end{prop}
\begin{proof}
We shall prove the conclusion by showing that $(X^m, G)$ is transitive for each $m\in \mathbb{N}$ by induction over $m\in \mathbb{N}$. Let $m\in \mathbb{N}$.
In fact, recalling that $G$ is infinite and observing \cite[Definition 3.2]{KerrLisoficindependence}, it suffices to prove that $IE_2 (X^m, G)= X^{2 m}$.
First we consider the case of $m= 1$. Observing
$X^2= IE^\Sigma_2 (X, G)= IE^\text{sof}_2 (X, G)$
by the assumption, $IE_2 (X, G)= X^2$ follows from \cite[Proposition 4.6]{KerrLisoficindependence}.
Now we assume that $IE_2 (X^k, G)= X^{2 k}$ for some $k\in \mathbb{N}$, we consider the case of
$m= k+ 1$. In fact, it follows directly from \cite[Theorem 3.3]{KerrLisoficindependence} that
$$IE_2 (X^{k+ 1}, G)= IE_2 (X^k, G)\times IE_2 (X, G)= X^{2 k}\times X^2= X^{2 (k+ 1)}.$$
This shows that $IE_2 (X^m, G)= X^{2 m}$ for each $m\in \mathbb{N}$, which finishes the proof.
\end{proof}

\begin{thm} \label{1307151844}
Assume that $G$ is an infinite countable discrete sofic group and $(X, G)$ has positive (sofic) topological entropy (with respect to some sofic approximation sequence). Then it admits a non-trivial weakly mixing subset of all orders.
\end{thm}
\begin{proof}
Applying \cite[Proposition 4.16 and Theorem 8.1]{KerrLisoficindependence} to the assumption, there exists a Cantor subset $K\subset X$ (and hence non-trivial) such that: for all $m\in \mathbb{N}$, once $x_1, \cdots, x_m\in K$ are distinct and $x_1', \cdots, x_m'\in K$ (need not be distinct), then
\begin{equation} \label{1308151532}
\liminf_{G\ni g\rightarrow \infty} \max_{1\le i\le m} \rho (g x_i, x_i')= 0.
\end{equation}
From this we have that $K$ is a non-trivial weakly mixing subset of all orders: let $m\in \mathbb{N}$, and let $U_1, V_1, \cdots, U_m, V_m$ be open subsets of $X$ intersecting $K$, obviously we could choose distinct points $x_i\in U_i\cap K$ and $x_i'\in V_i\cap K$ for each $i= 1, \cdots, m$ (observing that $K$ is a Cantor set), and then by \eqref{1308151532} we choose $g\in G$ such that $\max_{1\le i\le m} \rho (g x_i, x_i')$ is small enough and so $g x_i\in V_i$ for each $i= 1, \cdots, m$, that is, $g\in N (U_1\cap K, V_1)\cap \cdots\cap N (U_m\cap K, V_m)$, finishing the proof.
\end{proof}

\section*{Acknowledgements}

 The authors thank Wen Huang for sharing his joint work \cite{HuangXuYi} and for his valuable discussions, and thank Hanfeng Li for pointing out an error when the authors was proving a stronger result than Proposition \ref{1307262129} in a previous preprint version of the file and for his important comments about current statement of Proposition \ref{1307262129}.

 Zhang is supported by FANEDD (201018) and NSFC (11271078).

\vskip 16pt

\bibliographystyle{amsplain}


\begin{thebibliography}{10}

\bibitem{Ber}
Robert Berger, \emph{The undecidability of the domino problem}, Mem. Amer.
  Math. Soc. No. \textbf{66} (1966), 72.

\bibitem{BlanchardGlasnerKolyadaMaass}
Fran{\c{c}}ois Blanchard, Eli Glasner, Sergi{\u\i} Kolyada, and Alejandro
  Maass, \emph{On {L}i-{Y}orke pairs}, J. Reine Angew. Math. \textbf{547}
  (2002), 51--68.

\bibitem{BlanchardHostMaass}
Fran{\c{c}}ois Blanchard, Bernard Host, and Alejandro Maass, \emph{Topological
  complexity}, Ergodic Theory Dynam. Systems \textbf{20} (2000), no.~3,
  641--662.

\bibitem{BlanchardHostRuette}
Fran{\c{c}}ois Blanchard, Bernard Host, and Sylvie Ruette, \emph{Asymptotic
  pairs in positive-entropy systems}, Ergodic Theory Dynam. Systems \textbf{22}
  (2002), no.~3, 671--686.

\bibitem{BlanchardHuang}
Fran{\c{c}}ois Blanchard and Wen Huang, \emph{Entropy sets, weakly mixing sets
  and entropy capacity}, Discrete Contin. Dyn. Syst. \textbf{20} (2008), no.~2,
  275--311.

\bibitem{BowenJAMS}
Lewis Bowen, \emph{Measure conjugacy invariants for actions of countable sofic
  groups}, J. Amer. Math. Soc. \textbf{23} (2010), no.~1, 217--245.

\bibitem{BowenLi}
Lewis Bowen and Hanfeng Li, \emph{Harmonic models and spanning forests of
  residually finite groups}, J. Funct. Anal. \textbf{263} (2012), no.~7,
  1769--1808.

\bibitem{ChungLi}
Nhan-Phu Chung and Hanfeng Li, \emph{Homoclinic groups, {IE} groups and
  expansive algebraic actions}, Invent. Math. (2014),
  http://dx.doi.org/10.1007/s00222-014-0524-1, in press, arXiv:1103.1567.

\bibitem{CortezPetiteggd}
Mar{\'{\i}}a~Isabel Cortez and Samuel Petite, \emph{Invariant measures and
  orbit equivalence for generalized {T}oeplitz subshifts}, Groups Geom. Dyn.,
  to appear.

\bibitem{DeningerSchmidtetds}
Christopher Deninger and Klaus Schmidt, \emph{Expansive algebraic actions of
  discrete residually finite amenable groups and their entropy}, Ergodic Theory
  Dynam. Systems \textbf{27} (2007), no.~3, 769--786.

\bibitem{DooleyZhang}
Anthony~H. Dooley and Guohua Zhang, \emph{Local entropy theory of a random
  dynamical system}, Memoirs Amer. Math. Soc., vol. 233, no. 1099, Amer. Math. Soc.,
  Providence, RI, 2015,  pp.~1--120.


\bibitem{DouYeZhang}
Dou Dou, Xiangdong Ye, and Guohua Zhang, \emph{Entropy sequences and maximal
  entropy sets}, Nonlinearity \textbf{19} (2006), no.~1, 53--74.

\bibitem{EinsiedlerSchmidt}
Manfred Einsiedler and Klaus Schmidt, \emph{Markov partitions and homoclinic
  points of algebraic {$\bold Z^d$}-actions}, Tr. Mat. Inst. Steklova
  \textbf{216} (1997), no.~Din. Sist. i Smezhnye Vopr., 265--284.

\bibitem{Furstenberg}
Harry Furstenberg, \emph{Disjointness in ergodic theory, minimal sets, and a
  problem in {D}iophantine approximation}, Math. Systems Theory \textbf{1}
  (1967), 1--49.

\bibitem{Glasnerisrael}
Eli Glasner, \emph{A simple characterization of the set of {$\mu$}-entropy
  pairs and applications}, Israel J. Math. \textbf{102} (1997), 13--27.


\bibitem{Glasner05}
\bysame, \emph{Topological weak mixing and quasi-{B}ohr systems}, Israel J.
  Math. \textbf{148} (2005), 277--304, Probability in mathematics.

\bibitem{GlasnerThouvenotWeiss}
Eli Glasner, Jean-Paul Thouvenot, and Benjamin Weiss, \emph{Entropy theory
  without a past}, Ergodic Theory Dynam. Systems \textbf{20} (2000), no.~5,
  1355--1370.

\bibitem{HuangXuYi}
Wen Huang, Leiye Xu, and Yingfei Yi, \emph{Asymptotic pairs, stable sets and
  chaos in positive entropy systems}, J. Funct. Anal. \textbf{268}
  (2015), 824¨C-846.

\bibitem{HYZJFA2011}
Wen Huang, Xiangdong Ye, and Guohua Zhang, \emph{Local entropy theory for a
  countable discrete amenable group action}, J. Funct. Anal. \textbf{261}
  (2011), no.~4, 1028--1082.

\bibitem{KatokSchmidt95}
Anatole~B. Katok and Klaus Schmidt, \emph{The cohomology of expansive {${\bf
  Z}^d$}-actions by automorphisms of compact, abelian groups}, Pacific J. Math.
  \textbf{170} (1995), no.~1, 105--142.

\bibitem{KerrLiinvention}
David Kerr and Hanfeng Li, \emph{Entropy and the variational principle for
  actions of sofic groups}, Invent. Math. \textbf{186} (2011), no.~3, 501--558.


\bibitem{KerrLisoficindependence}
\bysame, \emph{Combinatorial independence and sofic entropy}, Commun. Math.
  Stat. \textbf{1} (2013), no.~2, 213--257.

\bibitem{KitchensSchmidt89}
Bruce Kitchens and Klaus Schmidt, \emph{Automorphisms of compact groups},
  Ergodic Theory Dynam. Systems \textbf{9} (1989), no.~4, 691--735.

\bibitem{LiYorke}
Tien~Yien Li and James~A. Yorke, \emph{Period three implies chaos}, Amer. Math.
  Monthly \textbf{82} (1975), no.~10, 985--992.

\bibitem{LindSchmidt}
Douglas Lind and Klaus Schmidt, \emph{Homoclinic points of algebraic {${\bf
  Z}^d$}-actions}, J. Amer. Math. Soc. \textbf{12} (1999), no.~4, 953--980.


\bibitem{LindSchmidtVerbitskiy}
Douglas Lind, Klaus Schmidt, and Evgeny Verbitskiy, \emph{Entropy and growth
  rate of periodic points of algebraic {$\Bbb Z^d$}-actions}, Dynamical
  numbers---interplay between dynamical systems and number theory, Contemp.
  Math., vol. 532, Amer. Math. Soc., Providence, RI, 2010, pp.~195--211.


\bibitem{LindSchmidtWard}
Douglas Lind, Klaus Schmidt, and Tom Ward, \emph{Mahler measure and entropy for
  commuting automorphisms of compact groups}, Invent. Math. \textbf{101}
  (1990), no.~3, 593--629.

\bibitem{LW}
Elon Lindenstrauss and Benjamin Weiss, \emph{Mean topological dimension},
  Israel J. Math. \textbf{115} (2000), 1--24.

\bibitem{Mycielski}
Jan Mycielski, \emph{Independent sets in topological algebras}, Fund. Math.
  \textbf{55} (1964), 139--147.

\bibitem{PZETDS}
Piotr Oprocha and Guohua Zhang, \emph{On local aspects of topological weak
  mixing, sequence entropy and chaos}, Ergodic Theory Dynam. Systems \textbf{34} (2014),
   1615--1639.

\bibitem{PZstudia}
\bysame, \emph{On local aspects of topological weak mixing in dimension one and
  beyond}, Studia Math. \textbf{202} (2011), no.~3, 261--288.

\bibitem{PZtopoappli}
\bysame, \emph{On sets with recurrence properties, their topological structure
  and entropy}, Topology Appl. \textbf{159} (2012), no.~7, 1767--1777.


\bibitem{PZadvance}
\bysame, \emph{On weak product recurrence and synchronization of return times},
  Adv. Math. \textbf{244} (2013), 395--412.

\bibitem{Rob}
Raphael~M. Robinson, \emph{Undecidability and nonperiodicity for tilings of the
  plane}, Invent. Math. \textbf{12} (1971), 177--209.

\bibitem{Schmidt95}
Klaus Schmidt, \emph{The cohomology of higher-dimensional shifts of finite
  type}, Pacific J. Math. \textbf{170} (1995), no.~1, 237--269.

\bibitem{SchmidtVerbitskiy}
Klaus Schmidt and Evgeny Verbitskiy, \emph{Abelian sandpiles and the harmonic
  model}, Comm. Math. Phys. \textbf{292} (2009), no.~3, 721--759.

\bibitem{Weiss}
Benjamin Weiss, \emph{A survey of generic dynamics}, Descriptive set theory and
  dynamical systems ({M}arseille-{L}uminy, 1996), London Math. Soc. Lecture
  Note Ser., vol. 277, Cambridge Univ. Press, Cambridge, 2000, pp.~273--291.


\bibitem{W0}
\bysame, \emph{Monotileable amenable groups}, Topology, ergodic theory, real
  algebraic geometry, Amer. Math. Soc. Transl. Ser. 2, vol. 202, Amer. Math.
  Soc., Providence, RI, 2001, pp.~257--262.

\end{thebibliography}

\end{document}